\documentclass{amsart}
\usepackage[utf8x]{inputenc}
\usepackage{amsmath, amsthm, amssymb}
\usepackage[usenames,dvipsnames,svgnames,table]{xcolor}
\usepackage[margin=1.3in]{geometry}

\newtheorem{theorem}{Theorem}[section]

\newtheorem{lemma}[theorem]{Lemma}

\newtheorem*{remark}{Remark}

\newcommand{\R}{\mathbb R}

\newcommand{\eps}{\varepsilon}
\newcommand{\e}{\eps}

\newcommand{\sech}{\ \mbox{sech}}
\newcommand{\dd}{\, \mathrm{d}}

\newcommand{\mb}{\mathbf}

\numberwithin{equation}{section}

\makeatletter
\@namedef{subjclassname@2020}{%
  \textup{2020} Mathematics Subject Classification}
\makeatother

\title[Stability analysis for variable-coefficient scalar field equations]{Linear and orbital stability analysis for solitary-wave solutions of variable-coefficient scalar field equations}

\author{Mashael Alammari}
\address{Department of Mathematical Sciences, Florida Institute of Technology, Melbourne, FL 32901}
\email{malammari2012@my.fit.edu}

\author{Stanley Snelson}
\address{Department of Mathematical Sciences, Florida Institute of Technology, Melbourne, FL 32901}
\email{ssnelson@fit.edu}

\keywords{Scalar-field equations, solitary waves, variable coefficients, spectral perturbation}
\subjclass[2020]{35L71, 35C07, 35P99}

\begin{document}

\maketitle

\begin{abstract}
We study general semilinear scalar-field equations on the real line with variable coefficients in the linear terms. These coefficients are uniformly small, but slowly decaying, perturbations of a constant-coefficient operator. We are motivated by the question of how these perturbations of the equation may change the stability properties of kink solutions (one-dimensional topological solitons). We prove existence of a stationary kink solution in our setting, and perform a detailed spectral analysis of the corresponding linearized operator, based on perturbing the linearized operator around the constant-coefficient kink. We derive a formula that allows us to check whether a discrete eigenvalue emerges from the essential spectrum under this perturbation. Known examples suggest that this extra eigenvalue may have an important influence on the long-time dynamics in a neighborhood of the kink. We also establish orbital stability of solitary-wave solutions in the variable-coefficient regime, despite the possible presence of negative eigenvalues in the linearization. 
\end{abstract}

\section{Introduction}

We consider a semilinear, variable-coefficient scalar field equation of the form
\begin{equation}\label{e:main}
\partial_t^2 u - \left[a(x)\partial_x^2 u +b(x)\partial_x u + c(x) u\right] +  F'(u) = 0, \quad x\in \R.
\end{equation}
Our assumptions on the potential $F$ are
\begin{equation}\label{e:F}
 \begin{split} F \in C^3(\R), \quad F(a_-) = F(a_+) = 0 \text{ for some } a_- < a_+,\\ F'(a_\pm) = 0, \quad F''(a_\pm) =m^2 >0, \quad F(s) > 0, s\in (a_-,a_+). \end{split}
 \end{equation}
 The linear operator $a(x)\partial_x^2 + b(x)\partial_x + c(x)$ is assumed to be a perturbation of the $1D$ Laplacian $\partial_x^2$. More precisely, for a small parameter $\delta>0$, we assume
\begin{equation}\label{e:abc}
\||a-1| + |\partial_xa| + |b| + |c|\|_{L^1(\R)} + \||a-1| + |\partial_x a| + |b| + |c|\|_{L^\infty(\R)} \leq \delta.
\end{equation}

With $\omega(x)= \exp(\int_{-\infty}^x b(z)/a(z) \dd z)$, the energy functional
\[ E(u) := \int_\R \frac {\omega(x)}{a(x)}\left(\frac 1 2 (\partial_t u)^2 + \frac 1 2 a (\partial_x u)^2 - \frac 1 2 cu^2 + F(u)\right) \dd x,\]
is formally conserved under the flow of \eqref{e:main}. We note that equation \eqref{e:main} is not invariant under translations, and that we make no parity assumptions on $F$ or the coefficients $a$, $b$, and $c$.

We are interested in the long-time behavior of solutions to \eqref{e:main}. Our first result (Theorem \ref{t:exist}) is the existence of a stationary solution $T$ of \emph{kink} or \emph{solitary-wave} type, i.e. an increasing stationary solution with $T(x) \to a_{\pm}$ as $x\to \pm\infty$. Standard arguments then show that \eqref{e:main} is locally well-posed for initial data $(u,\partial_t u)|_{t=0} \in H^1_T(\R)\times L^2(\R)$, where $H_T^1(\R) = \{\varphi: \varphi-T\in H^1(\R)\}$. In this context, $H^1_T(\R)\times L^2(\R)$ is referred to as the \emph{energy space}, and indeed, it is not hard to see (using in particular that $|\omega/a - 1|\lesssim \delta$) that functions in this space have finite energy. Our goal is to study the stability of $T$ with respect to small perturbations in the energy space.


\subsection{Motivation} 

One-dimensional kinks such as $T(x)$ are the simplest examples of topological solitons, and thus are an important model for physical phenomena arising in areas such as quantum field theory, condensed matter physics, and cosmology, among others. (See \cite{kinks-domainwalls, lohe,khare,top-sol} for some physics-oriented discussions.) Understanding their stability has proven to be a difficult mathematical challenge. 
The majority of work focuses on the constant-coefficient version of \eqref{e:main},
\begin{equation}\label{e:constant-speed}
 \partial_t^2 u - \partial_x^2 u  + F'(u) = 0.
 \end{equation}
In this constant-coefficient regime, it is standard that the assumptions \eqref{e:F} imply the  existence of a kink solution connecting $a_-$ and $a_+$. (Convenient proofs of this fact may be found in \cite[Lemma 1.1]{KMMV2020kink} or \cite[Proposition 2.1]{jendrej2019kink-antikink}.) This constant-coefficient stationary kink, which we denote by $S$, satisfies 
\begin{equation}\label{e:S}
 -S''  +F''(S) = 0, \quad \lim_{x\to -\infty} S(x) = a_-, \lim_{x\to \infty}S(x) = a_+.
 \end{equation}
We find in Theorem \ref{t:exist} that $T$ and $S$ are close in an appropriate norm.

 Orbital stability of $S$ in the constant-coefficient setting has been known for some time \cite{henry-perez-wreszinski}, and we extend this to our setting in Theorem \ref{t:orbital}. Asymptotic stability of kinks is more subtle, and depends on the specific choice of potential $F$. In particular, the two most studied versions of \eqref{e:constant-speed} are the $\phi^4$ equation with nonlinearity $F'(u) = u^3-u$, which is known to be asymptotically stable with respect to odd perturbations \cite{KMMphi4} and conjectured to be asymptotically stable in general; and the sine-Gordon equation with nonlinearity $F'(u) =\sin(u)$, which is not asymptotically stable, at least with respect to perturbations in the energy space. (See Section \ref{s:examples} for more on these examples.)

Our motivation is to understand the effect of \emph{linear} perturbations of the equation \eqref{e:constant-speed} on the stability properties of kink solutions. On the one hand, given that \eqref{e:constant-speed} is in some sense an idealized model, it is important on physical grounds to understand whether stability properties of kink solutions persist under perturbations of the equation. 
There is also reason to expect such perturbations to have a nontrivial qualitative impact on the stability analysis (rather than simply adding a small error term) in some situations. 
 As we explain below, this is connected with the possibility that a discrete eigenvalue may emerge from the essential spectrum of the linearized operator around the kink. 

\subsection{Main results}

 Before stating our main theorems, we make a technically convenient change of variables in \eqref{e:main} that will be in effect throughout this article.  Letting $y = \int_0^x a^{-1/2}(z) \dd z$, and abusing notation by writing $u(t,y) =  u(t,x(y))$ and $b(y) = b(x(y)) - a^{-1/2}(x(y))\frac d {dy} a^{1/2}(x(y))$, we have 
\begin{equation}\label{e:main-y}
\partial_t^2 u -[\partial_y^2 u + b(y)\partial_y u + c(y)u] + F'(u) = 0.
\end{equation}
The hypotheses \eqref{e:abc} imply
\begin{equation}\label{e:bc}
 \||b| + |c|\|_{L^1(\R)} + \||b| + |c|\|_{L^\infty(\R)} \leq C_0\delta,
 \end{equation}
 for some $C_0>0$.
 
 Our first result is the existence of a stationary kink:
 \begin{theorem}\label{t:exist}
 Assume that $0$ is not an $L^2(\R)$-eigenvalue of the operator $-\partial_y^2 - b\partial_y - (c-F''(S))$. Then, for $\delta>0$ sufficiently small, there exists a solution $T$ to
\begin{equation}\label{e:T}
-T'' -  b(y)T' - c(y)T + F'(T) = 0, \quad \lim_{y\to -\infty} T(y) = a_-, \lim_{y\to \infty}T(y) = a_+.
 \end{equation}
This solution can be written $T(y) = S(y) + S_b(y)$, where $S$ solves \eqref{e:S} and
\[ \|S_b\|_{W^{1,1}(\R)} + \|S_b\|_{W^{1,\infty}(\R)} \leq C\delta.\]
\end{theorem}
Unlike $S$, which satisfies $|S(x)-a_\pm|\lesssim e^{\mp mx}$ and $|S'(x)|\lesssim e^{-m|x|}$, our static kink $T$ does not necessarily posess exponential tails. This behavior is reminiscent of some higher-order, constant-coefficient field theories that do not fit into the assumptions \eqref{e:F} (see e.g. \cite{higher-order-power-law}). Under additional exponential decay assumptions on $b$ and $c$, it is possible to show $T$ has exponential asymptotics at $\pm\infty$ as in \cite[Theorem 1.1]{snelson2016stability}, but we do not explore the details here.
 
Our next result concerns the linearized operator around $T$. Writing $u(t,y) = T(y) + \varphi(t,y)$, the perturbation $\varphi$ satisfies
\[ \partial_t^2 \varphi - \partial_y^2\varphi- b\partial_y \varphi - c\varphi =F'(T) - F'(T+\varphi).\]
Adding $F''(T)\varphi$ to both sides, and defining the linear operator $\mathcal L_T = -\partial_y^2 - b\partial_y - c + F''(T)$ and the nonlinearity $\mathcal N(T,\varphi) = F'(T) - F'(T+\varphi) + F''(T)\varphi = O(\varphi^2)$, the equation for $\varphi$ can be written as a nonlinear Klein-Gordon equation:
\begin{equation}\label{e:nlkg}
 \partial_t^2 \varphi + \mathcal L_T \varphi = \mathcal N(T,\varphi).
 \end{equation}
We are most interested in situations where the spectrum of $\mathcal L_S= - \partial_y^2 + F''(S)$, the operator corresponding to the constant-coefficient kink, is known exactly. We then have $\mathcal L_T = \mathcal L_S - b\partial_y- c +F''(T) - F''(S)$, and we ask how the perturbation $-b\partial_y - c + F''(T) - F''(S)$ changes the spectral properties of $\mathcal L_S$.  


The $L^2(\R)$ spectrum of $\mathcal L_S$ is given by
\[ \sigma(\mathcal L_S) = \{0, \lambda_1,\ldots,\lambda_n\} \cup [m^2,\infty),\]
where $\lambda_1,\ldots,\lambda_n$ is a possibly empty, increasing collection of positive, simple eigenvalues. The eigenfunction corresponding to $0$ is exactly $S'$, the translation invariance mode.


As expected, discrete eigenvalues $\lambda_i$ will drift to nearby discrete eigenvalues $\lambda_i'$ of $\mathcal L_T$ under the perturbation. A more delicate question is whether an extra discrete eigenvalue emerges from the essential spectrum. This aspect is especially relevant when $\mathcal L_S$ has a threshold resonance, i.e. a function $R\in L^\infty(\R)\setminus L^2(\R)$ satisfying $\mathcal L_S R = m^2 R$, as is the case for both the $\phi^4$ and sine-Gordon equations. We derive a criterion in terms of $R$ and the coefficients $b$ and $c$ that governs whether the resonance drifts into a discrete eigenvalue.

Our results on the spectrum of $\mathcal L_T$ are collected in the following theorem:

\begin{theorem}\label{t:spectrum} Let $\mathcal L_T$ and $\mathcal L_S$ be as defined above. There exists a universal $c_0>0$ such that:
\begin{enumerate}
\item[(a)] The spectrum $\sigma (\mathcal L_T)$ is real, the essential spectrum $\sigma_{ess}(\mathcal L_T) = \sigma_{ess}(\mathcal L_S) = [m^2,\infty)$, and $\sigma(\mathcal L_T)$ lies in the $c_0\delta$-neighborhood of $\sigma(\mathcal L_S)$.

\item[(b)] For every eigenvalue $\lambda\in \sigma_d(\mathcal L_S)$ with eigenvector $Y_\lambda$, there is a corresponding $\lambda' \in \sigma_d(\mathcal L_T)$. The eigenvalue $\lambda'$ is real, simple, and satisfies $|\lambda - \lambda'| \leq c_0\delta$.  Also, if 
\[A:= \int_\R Y_\lambda [(F''(T) - F''(S)-c) Y_\lambda - b\partial_yY_\lambda] \dd y \neq 0,\]
then $\lambda' - \lambda$ has the same sign as $A$. The eigenfunction $Y_{\lambda'}$ of $\mathcal L_T$ corresponding to $\lambda'$ satisfies $|Y_{\lambda'}(y)| + |Y_{\lambda'}'(y)|\lesssim e^{-\sqrt{m^2-\lambda'}|y|}$. Furthermore, for suitable normalizations of $Y_\lambda$ and $Y_{\lambda'}$, we have
\[ \|e^{\sqrt{m^2 - \lambda'}|y|}Y_{\lambda'}(y) - e^{\sqrt{m^2-\lambda}|y|} Y_{\lambda}(y)\|_{L^\infty(\R)} \leq C\delta,\]
for a universal constant $C>0$.

\item[(c)] If $m^2$ is a simple resonance of $\mathcal L_S$, and 
\[ \int_\R R[(F''(T) - F''(S)-c) R - b\partial_y R] \dd y <0,\]
then there exists a discrete eigenvalue $\lambda$ of $\mathcal L_T$ with $0< m^2-\lambda < c_0\delta$. The eigenfunction $Y_{\lambda}$ also satisfies $|Y_{\lambda}(y)| + |Y_{\lambda}'(y)|\lesssim e^{-\sqrt{m^2-\lambda}|y|}$ and
\[ \|Y_{\lambda}(y) -  e^{-\sqrt{m^2 - \lambda}|y|}R(y)\|_{L^\infty(\R)} \leq C(k+\delta)e^{-\sqrt{m^2 - \lambda}|y|},\]
for suitable normalizations of $Y_\lambda$ and $R$.

If  
\[ \int_\R R[(F''(T) - F''(S)-c) R - b\partial_y R] \dd y >0,\]
then there are no eigenvalues of $\mathcal L_T$ in $[m^2-c_0\delta, m^2]$, and $m^2$ is non-resonant.

\item[(d)] If $\lambda=m^2$ is not a resonance or an embedded eigenvalue of $\mathcal L_S$, then the same is true of $\mathcal L_T$, and there are no eigenvalues of $\mathcal L_T$ in $[m^2-c_0\delta, m^2]$.
\end{enumerate}
\end{theorem}

Part (a) of this theorem is standard, and included for clarity of exposition. Part (b) is arguably not surprising, but its proof (see Section \ref{s:spectrum}) is a useful warm-up for parts (c) and (d). We also remark that the formulas in this theorem may be replaced with (more cumbersome, but in some sense more elementary) formulas that depend only on $F$, $S$, $b$, and $c$, via a first-order approximation for $F''(T) - F''(S)$. (See \eqref{e:formula1} and \eqref{e:formula2}.)

The possible extra eigenvalue as in Theorem \ref{t:spectrum}(c) is one of our primary motivations for performing this perturbation analysis. In general, eigenvalues lying in between $0$ and $m^2$ have a profound impact on the stability properties of the kink. At the very least, any proof of asymptotic stability or instability for $T$ would likely need to account for this extra eigenvalue in some way. 

 It should be noted that we are outside the realm of analytic perturbation theory, since we do not assume any continuity of the coefficients $b,c$ with respect to $\delta$. Our spectral analysis is based on the well-known method of finding solutions $U_{\pm\infty}^\lambda$ to the eigenvalue equation $\mathcal L_T U^\lambda = \lambda U^\lambda$ which decay at $\pm\infty$, and studying the Evans function (see e.g.  \cite{evans1974evans-function, jones1984evans-function, pego-weinstein, kapitula2002evans, kapitula2004evans}) which is related to the Wronskian of $U_\infty^\lambda$ and $U_{-\infty}^\lambda$. The key property is that the Wronskian is zero when $\lambda$ is an eigenvalue or resonance of $\mathcal L_T$. The slow decay of our coefficients $b$ and $c$ (as well as $F''(T) - F''(S)$) rules out tools such as the Gap Lemma (see \cite{alexander1990evans, gardner1998gaplemma}) which would allow one to analytically continue the Evans function past the threshold $\lambda = m^2$, but which requires exponential decay of the coefficients. 




Our last main result establishes the orbital stability of $T$:

\begin{theorem}\label{t:orbital}
There exists an $\eps>0$, depending on $\delta$, such that for any initial data $(u, \partial_t u)\Big|_{t=0} = (T + v_1,v_2)$ for $(v_1,v_2)\in L^2\times H^1$ with
\[ \|(v_1,v_2)\|_{H^1(\R)\times L^2(\R)} < \eps,  \]
the corresponding solution $u$ to \eqref{e:main-y} exists globally in time, and satisfies
\[ \|u-T\|_{H^1(\R)} + \|\partial_t u\|_{L^2(\R)} \leq C\eps,\]
for some $C$ depending on $\delta$ and $F$.
\end{theorem}

The proof is based on classical energy arguments, but must contend with the lack of translation invariance.


\subsection{Examples}\label{s:examples}

\subsubsection{$\phi^4$ model} The choice of a double-well potential $F(u) = \frac 1 4 (1-u^2)^2$ leads to the $\phi^4$ model
\begin{equation}\label{e:phi4}
 \partial_t^2 u - \partial_x^2 u = u - u^3.
 \end{equation}
Standard references on this equation include \cite{segur1987, cuccagna2008kink, KMMphi4, ross2019}. In this case, the kink solution $S(x) = \tanh(x/\sqrt2)$ is known explicitly, and the linearization $\mathcal L_S = -\partial_x^2 +(3S^2 - 1)$ has spectrum equal to 
\[ \sigma(\mathcal L_S) = \left\{0, \frac 3 2\right\}\cup [2,\infty).\]
The odd eigenfunction $Y_{3/2} = \tanh(x/\sqrt 2)\sech(x/\sqrt 2)$ corresponding to $\lambda = \frac 3 2$ is known as the \emph{internal oscillation mode}. The operator $\mathcal L_S$ also possesses an even resonance $R = 2\tanh^2(x/\sqrt 2)- \sech^2(x/\sqrt 2)$ at the threshold $\lambda = 2$. 

The $\phi^4$ kink is asymptotically stable with respect to \emph{odd} perturbations in the energy space, by the important work of Kowalczyk-Martel-Mu{\~n}oz \cite{KMMphi4}.  When working in the odd energy space, the even translation invariance mode at $\lambda=0$ and the even resonance do not play any role, but the internal oscillation mode has a dramatic effect on the dynamics.  The method of \cite{KMMphi4} involved projecting $\varphi$ onto $Y_{3/2}$ and the continuous spectrum, and carefully tracking the interaction between these two parts induced by the nonlinear terms of \eqref{e:nlkg}. A delicate coupling between the internal oscillation mode and the continuous part leads to the dissipation of energy away from a neighborhood of the kink. 

Asymptotic stability with respect to odd perturbations was extended to a variable-coefficient version of \eqref{e:phi4} by the second named author in \cite{snelson2016stability}, though the coefficients were less general than those considered here (only a second-order perturbation, which was taken to be even and exponentially decaying). The symmetry assumption means that any eigenvalue emerging from the essential spectrum would be even, and therefore can be ignored.

It remains an important open question whether this kink is asymptotically stable with respect to general perturbations. Our Theorem \ref{t:spectrum} implies that for certain choices of $b, c$ in \eqref{e:main-y}, the bottom of the continuous spectrum is non-resonant and there are no extra discrete eigenvalues. Such a version of \eqref{e:main-y} could serve as an interesting test case for the $\phi^4$ asymptotic stability problem, especially if one is convinced that the threshold resonance is an important source of difficulties.

\subsubsection{Sine-Gordon equation}

The choice $F(u) = 1-\cos(u)$ results in the sine-Gordon equation:
\[ \partial_t^2 u - \partial_x^2 u = -\sin(u).\]
This equation arises in the study of superconductivity as well as of surfaces with constant negative curvature, among other areas. (See e.g. \cite{ivancevic-sinegordon, cuenda2011sg, sine-gordon-book} for background on this equation.)

The explicit static kink is given by $S(x) = 4\arctan(e^x)$. The equation, which is completely integrable, possesses other special solutions including breathers and wobbling kinks \cite{cuenda2011sg, segur1983wobbles}. The presence of these wobbling kinks (periodic-in-time, spatially localized perturbations of the kink) implies that $S$ is not asymptotically stable in the energy space. (However, see \cite{chen2020sine-gordon} for an asymptotic stability result in a different topology, and \cite{AMP2020sinegordon}, which identified an infinite-codimensional manifold of initial data near the kink for which asymptotic stability in the energy space does hold.) With $\mathcal L_S = -\partial_x^2 + \cos(S)$ the linearization around $S$, it is known that
\[ \sigma(\mathcal L_S) = \{0\}\cup [1,\infty),\]
The failure of asymptotic stability in the energy space is consistent with the absense of an internal oscillation mode, which rules out the mechanism of stability observed for the $\phi^4$ model in \cite{KMMphi4}. However, there is an odd resonance $R(x) = \tanh(x)$ at the bottom of the continuous spectrum. Our Theorem \ref{t:spectrum} gives conditions under which the variable-coefficient version of sine-Gordon possesses a discrete eigenvalue $\lambda$ with $0 < 1-\lambda \ll 1$. In this case, one may ask whether the new odd eigenfunction behaves sufficiently like an internal oscillation mode that a stability mechanism like the one mentioned above comes into force. We plan to explore this question in a future article.

Somewhat different perturbed forms of the sine-Gordon equation have been considered in, e.g.,  \cite{denzler1993sinegordon, derks-singular, danna-sinegordon, fiore-sinegordon}. The general belief is that breathers and wobbles are non-generic phenomena, so one may conjecture that some dense set of coefficients satisfying \eqref{e:bc} lead to asymptotic stability.

\subsubsection{Other examples}

Let us briefly mention some other models whose variable-coefficient counterparts are included in our setting: the $P(\phi)_2$ theory \cite{lohe}, the double-sine-Gordon equation \cite{double-sine-gordon}, and certain higher-order field theories \cite{khare}, i.e. potentials equal to a polynomial of even degree, which in some cases satisfies the assumptions \eqref{e:F} and other cases not.

\subsection{Related work} 

The asymptotic stability of kinks in scalar field equations such as \eqref{e:constant-speed} is an active area of inquiry. In addition to the results mentioned above, we should mention the  recent work  of Kowalczyk-Martel-Mu{\~n}oz-Van Den Bosch \cite{KMMV2020kink}, which proved asymptotic stability for a general class of scalar-field models satisfying a condition on the potential $F$ that, in particular, rules out internal oscillation modes and threshold resonances. In the setting of odd perturbations, Delort-Masmoudi \cite{delort-masmoudi2020kink} established explicit decay rates for odd perturbations of the $\phi^4$ kink on time scales of order $\eps^{-4}$, where $\eps$ is the size of the initial perturbation. Let us also mention asymptotic stability results by Komech-Kopylova \cite{KKgl1,KKgl2} for kink solutions of relativistic Ginzburg-Landau equations, which are of the form \eqref{e:constant-speed} with additional assumptions of the flatness of $F$ at $a_\pm$.

This class of questions is a partial motivation for the closely related subject of scattering theory for NLKG equations similar to \eqref{e:nlkg}. See \cite{delort2001nlkg, delort2004nlkg, lindblad2015scattering,sterbenz2016decay,lindblad2019scattering,lindblad2020scattering, germain2020nlkg} and the references therein.

The operator $\mathcal L_S$ is (up to subtraction by $m^2I$) a Schr\"odinger operator with rapidly decaying potential. There is a well-established theory of spectral perturbation of Schr\"odinger and related operators, see e.g. the review \cite{simon2000review} for an overview. 
Works that specifically address perturbation of threshold resonances include \cite{jensen-melgaard, 2d-threshold, gh-resonances, rauch1980resonances}. As mentioned above, aspects such as the slow decay of coefficients and lack of continuous dependence on $\delta$ make it convenient to perform the perturbation ``by hand'' in our setting, rather than apply an abstract theorem or existing result.

\subsection{Outline of the paper} In Section \ref{s:stationary}, we prove the existence of the stationary solution $T$. In Section \ref{s:spectrum}, we perform a spectral perturbation analysis of the linearized operator around the kink, and in Section \ref{s:orbital}, we establish orbital stability of $T$.  Appendix \ref{s:a} contains some useful lemmas on the global solvability of second-order ODE systems.

\section{Stationary solution}\label{s:stationary}

First, we recall the existence of the static kink in the constant-coefficient case, which can be found by explicitly integrating the equation $S'' = F''(S)$. We quote from \cite[Lemma 1.1]{KMMV2020kink}:
\begin{lemma}\label{l:constant-speed}
Under the assumptions \eqref{e:F} on $F$, there is a solution $S\in C^4(\R)$ to the stationary equation
\[ -S'' + F'(S) = 0,\]
with $S'>0$ and $S\to a_{\pm}$ as $y\to \pm \infty$. Furthermore, $S$ and $S'$ satisfy
\[ |S(x) - a_\pm| \leq C e^{\mp m y}, \quad |S'(x)| \leq Ce^{-m|y|},\]
and the energy of $S$ is finite:
\[ \int_{\R} [S'(x)^2 + F(S(x))] \dd x < \infty.\]
\end{lemma}

We now prove the existence of a static kink $T(y)$ for our equation \eqref{e:main-y}:

\begin{proof}[Proof of Theorem \ref{t:exist}]
Let $S$ be the stationary solution to $-S'' + F'(S) = 0$ guaranteed by Lemma \ref{l:constant-speed}. Making the ansatz $T = S + S_b$, we have the following equation for $S_b$:
\begin{equation}\label{e:Sb}
 \begin{split}
 -S_b'' -  b S_b' - c S_b &= bS' + cS - F'(S+S_b) + F'(S) \\
&= bS' +cS - F''(S)S_b - \mathcal N(S,S_b),
\end{split}
\end{equation}
where $\mathcal N(S,S_b) = F'(S+S_b) -F'(S)  - F''(S)S_b$. 
Defining
\begin{equation*}
\mathcal{L}_b = -\partial_y^2 - b(y) \partial_y -c(y) + F''(S)(y) = \mathcal{L}_S - b(y) \partial_y - c(y),
\end{equation*} 
equation \eqref{e:Sb} becomes
\begin{equation}\label{e:Sb2}
\mathcal{L}_b S_b = b S' + cS - \mathcal N(S, S_b).
\end{equation}

We can find solutions $Y_{-\infty}, Y_\infty$ both satisfying $\mathcal L_b Y_{\pm \infty} = 0$, with $\lim_{y\to -\infty} Y_{-\infty} = 0$ and $\lim_{y\to \infty} Y_{\infty} = 0$. In more detail, $\mathcal L_b Y= 0$ may be written as the linear system $\mb Y' = (M_1 + M_2(y)) \mb Y$, with $\mb Y = (Y, Y')$, and
\[  M_1 = \left(\begin{array}{cc}0 & 1\\ m^2  & 0     \end{array}\right), \quad M_2(y) = \left(\begin{array}{cc}0 & 0\\ -c(y) + F''(S)(y) - m^2  & -b(y)     \end{array}\right). \]
Lemma \ref{l:system} below implies existence of $Y_\infty$ and $Y_{-\infty}$. 
In particular, $Y_\infty$ and $Y_{-\infty}$ are linearly independent, since otherwise there would be a nontrivial solution in $L^2$ to $\mathcal L_b Y = 0$, contradicting our assumption that $0$ is not an eigenvalue.

Define the Green's function
\[ G(y,w) := \frac 1 {W_{\mb Y}(y)} \begin{cases} Y_{-\infty}(y) Y_\infty(w), & y< w,\\
Y_{\infty}(y)Y_{-\infty}(w), & w\leq y,\end{cases}\]
where $W_{\mb Y}(y) = \det(\mb Y_{-\infty} ,\mb Y_{\infty})$. Abel's formula implies $W_{\mb Y}(y) = W_{\mb Y}(0)\exp(\int_{0}^y b(z) \dd z)$, which for $\delta>0$ sufficiently small, is bounded uniformly away from 0.  

For the inverse operator $\eta \mapsto \int_\R G(\cdot,w)\eta(w)\dd w$, we have the following useful bounds. First,
\begin{equation}\label{e:Linfbound}
\begin{split}
 \left|\int_\R G(y,w) \eta(w) \dd w\right|  &= \left|Y_\infty(y)\int_{-\infty}^y \frac {Y_{-\infty}(w)}{W_{\mb Y}(w)}\eta(w)\dd w + Y_{-\infty}(y)\int_{y}^\infty \frac {Y_{\infty}(w)}{W_{\mb Y}(w)} \eta(w)\dd w\right|\\
& \leq C \|\eta\|_{L^\infty(\R)}\left( e^{-my} \int_{-\infty}^y e^{m w} \dd w + e^{my}\int_y^\infty e^{-mw} \dd w\right)\\
&\leq C \|\eta\|_{L^\infty(\R)},
\end{split}
\end{equation}
for all $y\in \R$. We also have
\begin{equation}\label{e:L1bound}
\left|\int_\R\int_\R G(y,w) \eta(w) \dd w \dd y\right| \leq C\left(\int_\R e^{-my} \int_{-\infty}^y e^{mw} |\eta(w)| \dd w \dd y + \int_\R e^{my}\int_y^\infty e^{-mw} |\eta(w)| \dd w \dd y\right).
\end{equation}
For the first term on the right, we integrate by parts to obtain
\[ \int_\R e^{-my} \int_{-\infty}^y e^{mw} |\eta(w)| \dd w \dd y = \int_{-\infty}^y \frac{e^{m(w-y)}}{-m} |\eta(w)| \dd w \Big|_{y=-\infty}^{y=\infty} - \int_\R \frac {e^{-my}}{-m} e^{my} |\eta(y)| \dd y.\]
If $\eta\in L^1(\R)$, then since $e^{m(w-y)} \leq 1$, the boundary term at $-\infty$ vanishes, and the boundary term at $\infty$ is bounded by $\frac 1 m \|\eta\|_{L^1(\R)}$. After applying a similar calculation to the last term in \eqref{e:L1bound}, we conclude
\begin{equation}\label{e:L1}
 \left\|\int_\R G(\cdot,w)\eta(w)\dd w\right\|_{L^1(\R)} \leq C \|\eta\|_{L^1(\R)},
 \end{equation}
for a constant depending on $m$ and the coefficients $b,c$. The estimates \eqref{e:Linfbound} and \eqref{e:L1} clearly hold also if we replace $G(y,w)$ with $|G(y,w)|$. 

In addition, using $|Y_{\pm\infty}'(y)|\lesssim e^{\mp my}$, estimates similar to \eqref{e:Linfbound} and \eqref{e:L1} imply
\begin{equation}\label{e:derivative-bounds}
 \left\| \partial_y \int_{\R} G(y,w) \eta(w)\dd w \right\|_{L^1(\R)} \leq C\|\eta\|_{L^1(\R)}, \quad \left\| \partial_y \int_{\R} G(y,w) \eta(w)\dd w \right\|_{L^\infty(\R)} \leq C\|\eta\|_{L^\infty(\R)}.\end{equation}

Now we write the equation \eqref{e:Sb2} for $S_b$ as
\begin{equation}\label{e:integral}
 S_b(y) = (\mathcal T S_b)(y) := g(y) - \int_{\R} G(y,w) \mathcal N(S,S_b) \dd w,
 \end{equation}
where
\[ g(y) = \int_{\R} G(y,w) [b(w) S'(w) +c(w) S(w)]\dd w.\]

We want to find a fixed point for $\mathcal T$ in the space $X := L^1(\R)\cap L^\infty(\R)$ with norm $\|\cdot\|_{X} := \|\cdot\|_{L^1(\R)} + \|\cdot\|_{L^\infty(\R)}$. From \eqref{e:Linfbound} and \eqref{e:L1}, we have
\[ \|g\|_{X} \leq C \left(\|bS'+cS\|_{L^\infty(\R)}+ \|bS'+cS\|_{L^1(\R)}\right) \leq C_0 \delta,\]
since $S$ and $S'$ are bounded and $\|b+c\|_{X} \lesssim \delta$.
For the nonlinear term, since $F$ is $C^3$ on $[a_-, a_+]$, there is some $K>0$ such that 
\begin{equation}\label{e:taylor}
 |\mathcal N(S,\eta)| = |F'(S+\eta) - F'(S) - F''(S) \eta | \leq K \eta^2,
 \end{equation}
globally in $y$. This gives
\begin{equation}\label{e:N-norm}
\begin{split}
\left|\int_{\R} G(y,w) \mathcal N(S, \eta)(w) \dd w\right| &\leq K \int_{\R} |G(y,w)| \eta^2(w)\dd w
\end{split}
\end{equation}
and
\[ \begin{split}
\left|\int_{\R}\int_{\R} G(y,w) \mathcal N(S, \eta)(w) \dd w \dd y \right|&\leq K  \int_{\R}\left(e^{-my}\int_{-\infty}^y e^{mw} \eta^2(w)\dd w + e^{my}\int_{y}^\infty e^{-mw} \eta^2(w)\dd w\right)\dd y,
\end{split}\]
so that the estimates \eqref{e:Linfbound} and \eqref{e:L1} imply
\[ \left\| \int_\R G(y,w)\mathcal N(S,\eta)(w) \dd w \right\|_{X} \leq CK \|\eta^2\|_X \leq CK \|\eta\|_X^2,\]
after applying the standard interpolation $\|\cdot\|_{L^2}^2\leq \|\cdot\|_{L^\infty}\|\cdot\|_{L^1}$.

With $C_0$ such that $\|g\|_{X}\leq C_0 \delta$, define $\mathcal A := \{\eta \in X, \|\eta\|_{X} \leq 2C_0\delta\}$. For any $\eta \in \mathcal A$, the above estimates imply
\begin{align*}
\|\mathcal T \eta\|_{X} &= \left\|g - \int_{\R} G(\cdot,w) \mathcal N(S, \eta)(w)\dd w \right\|_{X} \leq \|g\|_{X} + CK\|\eta\|_{X}^2 \leq C_0\delta + C\delta^2, 
\end{align*}
so for $\delta < C_0/C$, we have $T\eta \in \mathcal A$. Next, for $\eta_1, \eta_2 \in \mathcal A$, we have from Taylor's Theorem that
\[ F'(S+\eta_1) = F'(S+\eta_2) + F''(S+\eta_1)(\eta_1-\eta_2) + \frac 1 2 F'''(\xi_y)(\eta_1-\eta_2)^2,\]
for some $\xi_y \in [a_-,a_+]$ depending on $y$. 
Using this in $\mathcal N(S,\eta_1) - \mathcal N(S,\eta_2)$, we have
\[
\begin{split}
 |\mathcal N(S,\eta_1) - \mathcal N(S,\eta_2)| &=| F'(S+\eta_1) - F'(S+\eta_2) - F''(S)(\eta_1-\eta_2)|\\
 &= \left|[F''(S+\eta_1) - F''(S)](\eta_1 - \eta_2) +\frac 1 2 F'''(\xi_y)(\eta_1 - \eta_2)^2\right|\\
 &\leq |\max_{[a_-,a_+]}|F'''(s)||\eta_1||\eta_1-\eta_2| + \frac 1 2 |F'''(\xi_y)|(\eta_1-\eta_2)^2\\
 &\leq K \delta|\eta_1 - \eta_2|,
 \end{split}
 \]
for some $K>0$. 
By \eqref{e:Linfbound} and \eqref{e:L1} we have
\begin{equation}\label{e:contraction}
\begin{split}
\|(\mathcal T \eta_1)(y) - (\mathcal T \eta_2)(y)\|_X &= \left\|\int_{\R} G(y,w) [\mathcal N(S, \eta_1) - \mathcal N(S,\eta_2)](w) \dd w\right\|_X\\
 &\leq CK \delta \|\eta_1-\eta_2\|_X,
\end{split}
\end{equation}
as above. 
The constant $CK>0$ depends on $m$ and the $C^3$ norm of $F$. For $\delta$ sufficiently small, we conclude $\mathcal T$ is a contraction on $\mathcal A$, and a unique solution $S_b$ to \eqref{e:integral} exists in $\mathcal A$.

To derive the bounds on $S_b'$, we differentiate equation \eqref{e:integral} and use the derivative bounds \eqref{e:derivative-bounds} and the Taylor estimate \eqref{e:taylor}:
\[ \begin{split} \|S_b'\|_{X} &= \left\|\partial_y \int_\R G(y,w) [ bS'+cS-N(S,S_b)]\dd w\right\|_{X}
\leq  \| bS'+cS-N(S,S_b)\|_{X} \lesssim \delta + K\|S_b^2\|_X \lesssim \delta. 
\end{split} \]
\end{proof}

The proof of Theorem \ref{t:exist} also provides the following approximation for $S_b = T-S$: since $S_b = g + \int_{\R} G(y,w) \mathcal N(S,S_b)(w) \dd w$, the estimates \eqref{e:Linfbound}, \eqref{e:L1} imply
\begin{equation}\label{e:Sb-1}
\begin{split}
\left\|S_b - \int_\R \tilde G(\cdot,w) [ bS'+cS](w)\dd w \right\|_{X} &\leq \left\| \int_\R \tilde G(\cdot,w) \mathcal N(S,S_b)(w)\dd w\right\|_{X}\\
&\lesssim \|S_b\|_{X}^2 \lesssim \delta^2,
\end{split}
\end{equation}
with $\tilde G = G$ if $0$ is not an eigenvalue of $\mathcal L_b = \mathcal L_S - b\partial_y - c$, and $\tilde G = G_\lambda$ otherwise, with $\lambda$ chosen such that $|\lambda|\lesssim \delta$, so that $\lambda S_b - \lambda\int_\R G_\lambda(y,w)S_b(w)\dd w$ are $O(\delta^2)$.

\section{Perturbation of the spectrum}\label{s:spectrum}

We consider the spectrum of
\begin{equation}\label{e:L}
\mathcal L_T := -\partial_y^2 - b\partial_y  - c + F''(T),
\end{equation}
where $T$ is the stationary solution guaranteed by Theorem \ref{t:exist}. Defining $d = - c + F''(T) - F''(S)$, we have $\mathcal L_T = \mathcal L_S - b \partial_y + d$. By the $C^3$ regularity of $F$, we have $|F''(T) - F''(S)|\leq K |S_b|$, and Theorem \ref{t:exist} implies $\|d\|_{L^1(\R)} + \|d\|_{L^\infty(\R)} \lesssim \delta$. With \eqref{e:Sb-1}, we can also write a first-order approximation for $d$ as follows:
\begin{equation}\label{e:d-expansion}
 d(y) = -c(y)  +F'''(S)(y) \int_\R \tilde G(y,w) [bS'+cS](w)\dd w + \eps(y),
 \end{equation}
with $\tilde G$ as in \eqref{e:Sb-1} and $\eps(y) = o(S_b(y))$.

Our goal is to investigate how the spectrum of $\mathcal L_S$ changes under the perturbation $- b\partial_y + d$. Since $\mathcal L_T$ is self-adjoint with respect to the inner product
\[ \langle f,g \rangle_\omega := \int_\R \omega f g \dd y,\]
with $\omega = \int_{-\infty}^y b(z) \dd z$, the spectrum $\sigma(\mathcal L_T)$ is real. Since the perturbation is relatively $\mathcal L$-compact, we have $\sigma_{ess}(\mathcal L_T) = \sigma_{ess}(\mathcal L_S)$. (See, e.g. \cite[Chapter 14]{hislop-sigal}.) 
Given our upper bounds on $b$ and $d$, it is standard that $\sigma(\mathcal L_T)$ lies in the $c_0\delta$ neighborhood of $\sigma(\mathcal L_S)$, for some $c_0>0$. (An elementary argument to this effect can be found in the proof of Theorem 3.1 in \cite{snelson2016stability}.) This already establishes part (a) of Theorem \ref{t:spectrum}.

To analyze the eigenvalue problem, we write the equation $(\mathcal L_S-\lambda) Y^\lambda = 0$ in vector form:
\[ (\mb Y^\lambda)'(y) = \left( \left(\begin{array}{cc} 0 & 1 \\ m^2-\lambda & 0\end{array}\right)+ \left(\begin{array}{cc} 0 & 0 \\ F''(S) - m^2 & 0\end{array}\right)\right)\mb Y^\lambda(y).\]
For any $\lambda \leq m^2$, Lemma \ref{l:system}(a) implies there exist $Y_\infty^\lambda, Y_{-\infty}^\lambda\in L^\infty_{\rm loc}(\R)$ 
satisfying $(\mathcal L_S -\lambda)Y_{\pm\infty}^\lambda = 0$, and
\begin{equation}\label{e:boundary}
 \lim_{y\to \pm \infty} e^{\pm k y} \mb Y_{\pm \infty}^\lambda(y) = \left(\begin{array}{c} 1\\ \mp k\end{array}\right),
 \end{equation}
with $k = \sqrt{m^2-\lambda}$. For $\lambda < m^2$, we also obtain the integral representations
\begin{equation}\label{e:Y-lambda}
\begin{split}
 e^{ky}\mb Y_{\infty}^\lambda &= \left(\begin{array}{c} 1 \\ -k\end{array}\right) - \frac 1 2 \int_y^\infty  (F''(S) - m^2) Y_\infty^\lambda(w) e^{kw}\left(\begin{array}{c} (e^{2k(y-w)}-1)/k \\ e^{2k(y-w)} +1\end{array}\right)\dd w\\
e^{-ky}\mb Y_{-\infty}^\lambda &= \left(\begin{array}{c} 1 \\ k\end{array}\right) - \frac 1 2 \int_{-\infty}^y  (F''(S) - m^2) Y_{-\infty}^\lambda(w) e^{-kw}\left(\begin{array}{c} (e^{2k(y-w)}-1)/k \\ e^{2k(y-w)} +1\end{array}\right)\dd w .
\end{split}
 \end{equation}

For the operator $\mathcal L_T$, we similarly apply Lemma \ref{l:system}(a) with $V = F''(S) - m^2 + d$ to obtain $U_{\pm \infty}^\lambda$ solving $(\mathcal L_T-\lambda) U_{\pm\infty}^\lambda = 0$, with the same boundary conditions \eqref{e:boundary}, and for $\lambda< m^2$, 
\begin{equation}\label{e:U-lambda}
\begin{split}
 e^{ky}\mb U_{\infty}^\lambda &= \left(\begin{array}{c} 1 \\ -k\end{array}\right) - \frac 1 2 \int_y^\infty  [(F''(S) - m^2+d) U_\infty^\lambda - b(U_\infty^\lambda)'] e^{kw}\left(\begin{array}{c} (e^{2k(y-w)}-1)/k \\ e^{2k(y-w)} +1\end{array}\right)\dd w\\
e^{-ky}\mb U_{-\infty}^\lambda &= \left(\begin{array}{c} 1 \\ k\end{array}\right) - \frac 1 2 \int_{-\infty}^y [ (F''(S) - m^2+d) U_{-\infty}^\lambda - b (U_{-\infty}^\lambda)']e^{-kw}\left(\begin{array}{c} (e^{2k(y-w)}-1)/k \\ e^{2k(y-w)} +1\end{array}\right)\dd w .
\end{split}
 \end{equation}

First, we prove a suitable approximation lemma for $Y_{\pm \infty}^\lambda$ and $U_{\pm \infty}^\lambda$ for nearby values of $\lambda$:

\begin{lemma}\label{l:approx}
Assume $\|d\|_{L^1(\R)} + \|b\|_{L^1(\R)} \leq 1$. For any compact subset $B$ of $(-\infty,m^2)$, there exists a constant $C>0$ such that for  any $\lambda_1, \lambda_2\in B$, there holds
\[\begin{split}
\|e^{k_1 y}\mb Y_{\infty}^{\lambda_1} - e^{k_2 y} \mb U_{\infty}^{\lambda_2}\|_{L^\infty([0,\infty),\R^2)} &\leq C(|\lambda_1 - \lambda_2| + \|d\|_{L^1(\R)} + \|b\|_{L^1(\R)}),\\
\|e^{-k_1 y}\mb Y_{-\infty}^{\lambda_1} - e^{-k_2 y} \mb U_{-\infty}^{\lambda_2}\|_{L^\infty((-\infty,0],\R^2)} &\leq C(|\lambda_1 - \lambda_2| +  \|d\|_{L^1(\R)} + \|b\|_{L^1(\R)}),
\end{split}\]
where $k_i = \sqrt{m^2 - \lambda_i}$. 
\end{lemma}


\begin{proof}
We prove only the first estimate, as the second follows by a similar argument. 

From \eqref{e:Y-lambda} and \eqref{e:U-lambda} we have 
 \[
 \begin{split}
  e^{k_1 y} \mb Y_{\infty}^{\lambda_1}&(y) - e^{k_2y} \mb U_\infty^{\lambda_2}(y)\\
   &= \left(\begin{array}{c} 0 \\ k_2-k_1\end{array}\right) - \frac 1 2 \int_y^\infty \left[ (F''(S) - m^2) \left(\begin{array}{c} (e^{2k_1(y-w)}-1)/k_1 \\ e^{2k_1(y-w)} +1\end{array}\right)  Y_\infty^{\lambda_1} (w) e^{k_1w} \right.\\
&\qquad\qquad \left.- \left(\begin{array}{c} (e^{2k_2(y-w)}-1)/k_2 \\ e^{2k_2(y-w)} +1\end{array}\right) [(F''(S)-m^2 + d) U_\infty^{\lambda_2}(w) - b(U_\infty^{\lambda_2})' e^{k_2w}] \right]\dd w \\
  &= \mb J_1(y) + \mb J_2(y)\\
  &\quad  - \frac 1 2 \int_y^\infty (F''(S) - m^2) \left(\begin{array}{c} (e^{2k_1(y-w)}-1)/k_1 \\ e^{2k_1(y-w)} +1\end{array}\right) (e^{k_1w} Y_\infty^{\lambda_1}(w) - e^{k_2w} U_\infty^{\lambda_2}(w))\dd w,
 \end{split}
 \]
 where 
\[\begin{split}
 \mb J_1(y) &: =  \left(\begin{array}{c} 0 \\ k_2-k_1\end{array}\right)\\
 &\quad  - \frac 1 2 \int_y^\infty (F''(S) - m^2) e^{k_2w} U_\infty^{\lambda_2}(w)\left(\begin{array}{c} (e^{2k_1(y-w)}-1)/k_1 - (e^{2k_2(y-w)} - 1)/k_2\\ e^{2k_1(y-w)} - e^{2k_2(y-w)}\end{array}\right) \dd w\\
 \mb J_2(y) &:= \frac 1 2 \int_y^\infty \left(\begin{array}{c} (e^{2k_2(y-w)}-1)/k_2 \\ e^{2k_2(y-w)} +1\end{array}\right)[d U_\infty^{\lambda_2} - b(U_\infty^{\lambda_2})'] e^{k_2w} \dd w.
 \end{split} \]

Since $y-w \leq 0$, the mean value theorem applied to $x\mapsto e^{2x(y-w)}$ and $x\mapsto (e^{2x(y-w)}-1)/x$ implies, after a straightforward calculation,  the inequalities
\[ \left| \left(\begin{array}{c} (e^{2k_1(y-w)}-1)/k_1 - (e^{2k_2(y-w)} - 1)/k_2\\ e^{2k_1(y-w)} - e^{2k_2(y-w)}\end{array}\right) \right| \leq  C(1+|y-w|)|k_1-k_2|,\]
for a constant $C$ depending on $B$. Since $|F''(S) - m^2|\leq e^{-m|w|}$ and $e^{k_2w} U_\infty^{\lambda_2}(w)$ is uniformly bounded on $[0,\infty)$, we therefore have $\|\mb J_1\|_{L^\infty([0,\infty),\R^2)} \leq C|k_1-k_2|$.

For $\mb J_2$, since $e^{k_2 w} \mb U_\infty^{\lambda_2}$ is bounded on $[0,\infty)$, we have $\|\mb J_2\|_{L^\infty([0,\infty),\R^2)} \leq C\|d + b\|_{L^1(\R)}$, for a constant depending only on $k_2$.  

Define the integral kernel 
\[K(y,w) = -\frac 1 2 (F''(S) - m^2)\left(\begin{array}{cc} (e^{2k_1(y-w)}-1)/k_1 & 0\\ e^{2k_1(y-w)} +1 & 0\end{array}\right),\]
From the exponential decay of $F''(S) - m^2$ 
we see that
\[ \int_0^\infty \sup_{0<y<w} \|K(y,w)\| \dd w \]
is bounded by a constant depending only on $k_1$ and $m$.  Lemma \ref{l:volterra} then implies
\begin{equation}\label{e:Ylambda-Y0}
\|e^{k_1y}\mb Y_{\infty}^{\lambda_1}(y) - e^{k_2y} \mb U_\infty^{\lambda_2}(y)\|_{L^\infty([0,\infty),\R^2)} \leq C \|\mb J_1 + \mb J_2\|_{L^\infty([0,\infty),\R^2)} \leq C(|k_1-k_2| + \|d\|_{L^1(\R)} + \|b\|_{L^1(\R)}).
\end{equation}
Since $|k_1 - k_2| \leq C|\lambda_1 - \lambda_2|$ for a constant depending only on $K$, the proof is complete.
\end{proof}

Now, we are ready to derive a result that governs the direction in which eigenvalues of $\mathcal L_S$ drift under the perturbation:

\begin{theorem}\label{t:eigenvalue}
Assume $\|d\|_{L^1(\R)} + \|b\|_{L^1(\R)} \leq \delta \ll 1$. For any eigenvalue $\lambda_*< m^2$ of $\mathcal L_S$ with eigenfunction $Y_{*}$, there exists a simple, real eigenvalue $\lambda$ of $\mathcal L_T$ with  $|\lambda - \lambda_*| \leq C\delta$. Furthermore, we have the following expansion for $\lambda$:
\[ \lambda =  \lambda_* + \frac {\int_{\R} Y_* [dY_* - bY_*'] \dd y}{\int_{\R} (Y_*)^2 \dd y} + O(\delta^2).\]
In particular, if
\[ A := \int_{\R} Y_* [d Y_*- bY_*'] \dd y \neq 0,\]
then $\lambda - \lambda_*$ has the same sign as $A$.
\end{theorem}


\begin{remark}
Using the formula \eqref{e:d-expansion}, one can show that $A$ has the same sign as
\begin{equation}\label{e:formula1}
  \int_{\R} \left[Y_*^2(y)\left(-c(y) + F'''(S)(y) \int_\R \tilde G(y,w)[bS'+cS](w)\dd w\right) - b(y)Y_*'(y)Y_*(y)\right] \dd y    .
 \end{equation}
\end{remark}

\begin{proof}
With $Y_{\pm\infty}^{\lambda_*}$ solving \eqref{e:Y-lambda}, since $\lambda_*$ is a simple eigenvalue, we have $Y_{\pm\infty}^{\lambda_*} = c_\pm Y_*$, for constants $c_\pm$. Let $k_* = \sqrt{m^2-\lambda_*}$. From our construction, it is clear that $Y_*$ decays exponentially at a rate $Y_*(y) \lesssim e^{-k_* |y|}$.

For $\lambda$ near $\lambda_*$ 
let $k=\sqrt{m^2-\lambda}$ and let $U_{\pm \infty}^\lambda$ be the solutions to \eqref{e:U-lambda} as above. From Lemma \ref{l:approx} and our assumption that $\|d\|_{L^1(\R)} + \|b\|_{L^1(\R)} \lesssim \delta$, 
 we can write
 \begin{equation}\label{e:E12}
 \begin{split}
 \mb U_{\pm \infty}^\lambda(y) &= e^{\pm(k_*-k)y} \mb Y_{\pm\infty}^{\lambda_*}(y) + e^{\mp k y} \mb E_{\pm \infty}(y),
 \end{split}
 \end{equation}
 with $\|\mb E_{\infty}\|_{L^\infty([0,\infty),\R^2)}+\|\mb E_{-\infty}\|_{L^\infty((-\infty,0],\R^2)}\lesssim \delta$. 
Denote the Wronskian
\[ W_{\mb U}(\lambda,y) = \det(\mb U_\infty^\lambda(y), \mb U_{-\infty}^\lambda(y)).\] 
By Abel's formula, $\exp(\int_{-\infty}^y b(z) \dd z) W_{\mb U}(\lambda,y)$ is independent of $y$. We focus on $y=0$ and apply \eqref{e:E12} to obtain 
\begin{equation}\label{e:WU}
 \begin{split}
W_{\mb U}(\lambda,0) &= \det(\mb Y_\infty^{\lambda_*}(0) + \mb E_\infty(0), \mb Y_{-\infty}^{\lambda_*}(0)+\mb E_{-\infty}(0))\\
&= \det(\mb Y_\infty^{\lambda_*}(0), \mb E_{-\infty}(0))+ \det( \mb E_{\infty}(0),\mb Y_{-\infty}^{\lambda_*}(0)) + \det(\mb E_\infty(0), \mb E_{-\infty}(0))\\
&= \det(\mb Y_\infty^{\lambda_*}(0), \mb U_{-\infty}^\lambda(0)) + \det(\mb U_\infty^\lambda(0), \mb Y_{-\infty}^{\lambda_*}(0)) + O(\delta^2).
\end{split}
\end{equation}
In the second line, we used that $\mb Y_\infty^{\lambda_*}$ and $\mb Y_{-\infty}^{\lambda_*}$ are parallel, and in the last line, we used $\mb E_{\pm\infty}(0) = \mb U_{\pm\infty}^\lambda(0) - \mb Y_{\pm \infty}^{\lambda_*}(0)$ and $|\mb E_{\pm \infty}(0)| \lesssim \delta$. Since 
\[\det(\mb Y_{\pm\infty}^{\lambda_*}, \mb U_{\mp\infty}^\lambda)'(y) = Y_{\pm\infty}^{\lambda_*} \left[(d+\lambda_*-\lambda)  U_{\mp\infty}^\lambda - b (U_{\mp\infty}^\lambda)'\right],\]
we can use \eqref{e:E12} again to write
\[ 
\begin{split}
\det( \mb Y_\infty^{\lambda_*}(0), \mb U_{-\infty}^\lambda(0)) &= \int_{-\infty}^0 Y_\infty^{\lambda_*}\left[(d + \lambda_*-\lambda)  U_{-\infty}^\lambda - b (U_{-\infty}^\lambda)'\right] \dd w\\
& = \int_{-\infty}^0  e^{(k-k_*)w} Y_\infty^{\lambda_*}\left[(d+\lambda_*-\lambda)  Y_{-\infty}^{\lambda_*} - b (Y_{-\infty}^{\lambda_*})'\right]\dd w \\
&\quad  +\int_{-\infty}^0 e^{kw} Y_\infty^{\lambda_*}(w) [(d+\lambda_*-\lambda) E_{-\infty}^1 - b E_{-\infty}^2]\dd w,
\end{split}
\]
and 
\[ \begin{split}
\det(\mb U_\infty^\lambda,\mb Y_{-\infty}^{\lambda_*})(0) &= \int_{0}^\infty Y_{-\infty}^{\lambda_*}\left[(d+\lambda_*-\lambda)  U_{\infty}^\lambda - b (U_{\infty}^\lambda)'\right] \dd w\\
& = \int_0^\infty  e^{(k_*-k)w} Y_{-\infty}^{\lambda_*}\left[(d + \lambda_*-\lambda)  Y_{\infty}^{\lambda_*} - b (Y_{\infty}^{\lambda_*})'\right]\dd w \\
&\quad  +\int_0^\infty e^{-kw} Y_{-\infty}^{\lambda_*}(w) [(d+\lambda_*-\lambda) E_{\infty}^1 - b E_{\infty}^2]\dd w.
\end{split}
\]
Feeding these expressions into \eqref{e:WU}, we obtain
\[ 
W_{\mb U}(\lambda,0) = c_+c_-\int_{-\infty}^\infty  e^{(k_*-k)|w|} [(d+\lambda_*-\lambda)(Y_*(w))^2 - b Y_*(w)Y_*'(w)]\dd w + O(\delta^2).
\]
From the approximation $|e^{(k_*-k)|w|} - 1 | \leq |k_*-k||w| e^{|k_*-k||w|}$ and the exponential decay of $Y_*$ 
we have, for $\lambda$ an eigenvalue of $\mathcal L_T$,
 \[ 
 0 = e^{\int_{-\infty}^y b(z) \dd z} c_+ c_- \int_{-\infty}^\infty[(d+\lambda_*-\lambda)(Y_*(w))^2 - b Y_*(w)Y_*'(w)]\dd w + O(\delta^2),\]
 which implies the first-order expansion for $\lambda$ in the statement of the theorem.
\end{proof}

Now, we analyze the threshold resonance $R$, which is an $L^\infty$ function solving $\mathcal L_S R- m^2 R=0$. First, we prove a modified version of Lemma \ref{l:approx} for the borderline case $m^2$. This will be useful in tracking how a threshold resonance of $\mathcal L_S$ translates to the spectrum of $\mathcal L_T$. Writing $(\mathcal L_S- m^2)Y^\lambda = 0$ in vector form as above, Lemma \ref{l:system}(b) implies
\begin{equation}\label{e:Y-m}
  \mb Y_\infty^{m^2}(y) = \left(\begin{array}{c} 1\\  0\end{array}\right)  -  \int_{y}^\infty (F''(S) - m^2)Y_\infty^{m^2} \left(\begin{array}{c} y-w\\ 1 \end{array}\right)\dd w.
  \end{equation}

 \begin{lemma}\label{l:approx-r}
Assume $\|d\|_{L^1(\R)} + \|b\|_{L^1(\R)} \leq 1$. For any $\lambda_0 < m^2$, there exists a constant $C>0$ such that for  any $\lambda\in (\lambda_0,m^2)$, there holds
\[\begin{split}
\|\mb Y_{\infty}^{m^2} - e^{k y} \mb U_{\infty}^{\lambda}\|_{L^\infty([0,\infty),\R^2)} &\leq C(k + \|d\|_{L^1(\R)} + \|b\|_{L^1(\R)}),\\
\|\mb Y_{-\infty}^{m^2} - e^{-k y} \mb U_{-\infty}^{\lambda}\|_{L^\infty((-\infty,0],\R^2)} &\leq C(k +  \|d\|_{L^1(\R)} + \|b\|_{L^1(\R)}),
\end{split}\]
where $k = \sqrt{m^2 - \lambda}$.
\end{lemma}

\begin{proof}
The proof is similar to Lemma \ref{l:approx}, with the difference that $\mb Y_\infty^{m^2}$ satisfies the modified integral equation \eqref{e:Y-m}. From \eqref{e:Y-m} and \eqref{e:U-lambda}, we have
 \[
 \begin{split}
  \mb Y_{\infty}^{m^2}(y) - e^{ky} \mb U_\infty^{\lambda}(y) &= \left(\begin{array}{c} 0 \\ k\end{array}\right) - \frac 1 2 \int_y^\infty \left[ 2(F''(S) - m^2) \left(\begin{array}{c} y-w \\ 1\end{array}\right)  Y_\infty^{m^2} (w) \right.\\
&\qquad\qquad \left.- \left(\begin{array}{c} (e^{2k(y-w)}-1)/k \\ e^{2k(y-w)} +1\end{array}\right) [(F''(S)-m^2 + d) U_\infty^{\lambda}(w) - b(U_\infty^{\lambda})' e^{kw}] \right]\dd w \\
  &= \mb J_1(y) + \mb J_2(y) -  \frac 1 2 \int_y^\infty (F''(S) - m^2)\left(\begin{array}{c}  (e^{2k(y-w)} - 1)/k\\e^{2k(y-w)} + 1\end{array}\right) (Y_\infty^{m^2}(w) - e^{kw} U_\infty^{\lambda}(w))\dd w,
 \end{split}
 \]
 with  
 \[
 \mb J_1(y) : =  \left(\begin{array}{c} 0 \\ k\end{array}\right) - \frac 1 2 \int_y^\infty (F''(S) - m^2) Y_\infty^{m^2}(w)\left(\begin{array}{c} 2(y-w) - (e^{2k(y-w)} - 1)/k\\ 2 - (e^{2k(y-w)} + 1)\end{array}\right) \dd w,
 \]
 and $\mb J_2(y)$ defined as in the proof of Lemma \ref{l:approx}, with $\lambda$ replacing $\lambda_2$.
 
  We claim that $\|\mb J_1\|_{L^\infty([0,\infty),\R^2)} \leq C|k|$. Indeed, applying the mean value theorem to $f(x) = e^{2x(y-w)}$ gives
\[|f(k) - f(0)| \leq |k|\sup_{0<x<k} |f'(x)| \leq |k| 2|y-w|e^{2x(y-w)} \leq |k|2|y-w|,\]
 or $|e^{2k(y-w)} - 1|\leq 2|k||y-w|$. Next, Taylor's Theorem implies $f(k) = f(0) + f'(0)k + \eps$, with $|\eps| \leq \frac 1 2 k^2 \sup_{0<x<k} |f''(x)|$. We have $f'(0) = 2(y-w)$ and $f''(x) = 4(y-w)^2 e^{2x(y-w)} \leq 4(y-w)^2$, since $y-w<0$. This gives $e^{2k(y-w)} - 1 = 2k(y-w)  + \eps$, with  $|\eps|\leq 2k^2(y-w)^2$, or \[|(e^{2k(y-w)} - 1)/k -2(y-w)|  \leq 2 k(y-w)^2.\] Plugging these inequalities into the definition of $\mathbf J_1$ and using decay of $F''(S) - m^2$ gives the desired estimate.
 
 The same calculation as in the proof of Lemma \ref{l:approx} implies that 
 the boundedness property \eqref{e:mu} is satisfied for this integral equation, 
 and Lemma \ref{l:volterra} establishes the conclusion of the lemma.
\end{proof}

In the following theorem, we make the (mild) assumption that the limits at $\pm\infty$ of $R(y)$ are nonzero.

\begin{theorem}
\textup{(a)} Assume that $m^2$ is a simple resonance for $\mathcal L_S$, i.e. that there exists $R \in L^\infty(\R) \setminus L^2(\R)$ with $\mathcal L_S R= m^2 R$. Then there exists $\delta>0$ depending only on the function $R$, such that if $\|b\|_{L^1(\R)} + \|d\|_{L^1(\R)} \lesssim \delta$ and
\[ \int_{\R} R [dR - b R']\dd w <0,\]
then there exists a discrete eigenvalue $\lambda$ of $\mathcal L_T$ with $0 < m^2 - \lambda < C\delta$. If  
\[ \int_{\R} R [dR - b R']\dd w >0,\]
then there is no discrete eigenvalue of $\mathcal L_T$ in a neighborhood of the essential spectrum, i.e. the discrete spectrum $\sigma_d(\mathcal L_T)$ consists of the same number of eigenvalues as $\sigma_d(\mathcal L_S)$. 

\textup{(b)} On the other hand, if $m^2$ is nonresonant and not an eigenvalue of $\mathcal L_S$, then for $\delta$ is sufficiently small, $m^2$ cannot be a resonance or an eigenvalue of $\mathcal L_T$, and there is no eigenvalue of $\mathcal L_T$ in a neighborhood of the essential spectrum.
\end{theorem}
\begin{remark}
As above, using \eqref{e:d-expansion}, the quantity $\int_\R R [dR-bR']\dd w$ has the same sign as
\begin{equation}\label{e:formula2}
\int_{\R} \left[R^2(y)\left(-c(y) + F'''(S)(y) \int_\R \tilde G(y,w)[bS'+cS](w)\dd w\right) - b(y)R'(y)R(y)\right] \dd y    
\end{equation}
\end{remark}
\begin{proof}
With $Y_{\pm\infty}^{m^2}$ solving \eqref{e:Y-m}, we have $Y_{\pm\infty}^{m^2} = c_\pm R$, for constants $c_\pm$. 

Our first step is to analyze the unperturbed Wronskian $W_{\mb Y}(\lambda, y) = \det(\mb Y_\infty^\lambda, \mb Y_{-\infty}^\lambda)$. With $\lambda$ near $m^2$ and $k = \sqrt{m^2 - \lambda}$, by abuse of notation, we write $W_{\mb Y}(k,y) = W_{\mb Y}(m^2-k^2,y)$. Applying Lemma \ref{l:approx-r} with $b = d = 0$, we may write 
\[ \mb Y_{\pm\infty}^\lambda(y) = e^{\mp ky} \mb (Y_{\pm \infty}^{m^2}(y) + \mb E_{\pm \infty}), \quad y\in \R,\]
with $\|\mb E_\infty\|_{L^\infty([0,\infty),\R^2)} + \|\mb E_{-\infty}\|_{L^\infty((-\infty,0],\R^2)} \lesssim k$. By the equation satisfied by $Y_{\pm\infty}^\lambda$, the Wronskian $W_{\mb Y}(\lambda,y)$ is independent of $y$. Proceeding as in the proof of Theorem \ref{t:eigenvalue}, we have as before (see \eqref{e:WU})
\begin{equation}\label{e:WY-R}
 \begin{split}
W_{\mb Y}(k,0) &= \det(\mb Y_\infty^{m^2}(0), \mb Y_{-\infty}^\lambda(0)) + \det(\mb Y_\infty^\lambda(0), \mb Y_{-\infty}^{m^2}(0)) + O(k^2).
\end{split}
\end{equation}
A direct calculation shows $\det(\mb Y_{\pm\infty}^{m^2}, \mb Y_{\mp\infty}^\lambda)'(y) = (m^2-\lambda) Y_{\pm \infty}^{m^2} Y_{\mp\infty}^\lambda$, which gives, since $k^2 = m^2-\lambda$,
\begin{equation}\label{e:detW}
\begin{split}
 W_{\mb Y}(k,0) &= k^2 \int_{-\infty}^0 Y_{\infty}^{m^2} Y_{-\infty}^\lambda \dd w + k^2 \int_0^\infty Y_{-\infty}^{m^2} Y_\infty^\lambda \dd w + O(k^2)\\
 &= k^2 \int_{-\infty}^0 Y_\infty^{m^2} e^{kw} (Y_{-\infty}^{m^2} + E_{-\infty}^1) \dd w + k^2\int_0^\infty Y_{-\infty}^{m^2} e^{-kw} (Y_\infty^{m^2} +  E_\infty^1) \dd w + O(k^2)\\
 &= k^2 c_+ c_- \int_{-\infty}^\infty e^{-k|w|} R^2 \dd w\\
 &\quad   + k^2 \int_{-\infty}^\infty R(w) e^{-k|w|}  \left( 1_{\{w<0\}} c_+ E_{-\infty}^1 + 1_{\{w\geq 0\}} c_- E_{\infty}^1\right) \dd w  + O(k^2).
 \end{split}
 \end{equation}
Note that all integrals converge, since $Y_{\pm\infty}^{m^2}$, $e^{\pm ky} Y_{\pm\infty}^\lambda$, and $e^{\pm ky} E_{\pm \infty}^1$ are all uniformly bounded. 

In the last expression of \eqref{e:detW}, we note that the first term is proportional to $k$. Indeed, since $R$ has non-zero limits as $y\to \pm \infty$, there exist $\zeta,M>0$ (independent of $k$) such that $R^2(y)\geq \zeta$ if $|y|\geq M$. As a result, for any $k\in (0,1)$, one has $\int_{\R} e^{-k|w|} R^2 \dd w \geq 2\zeta e^{-M}/k$. On the other hand, we have $\int_{\R} R^2 e^{-k|w|}\dd w \leq 2\|R\|_{L^\infty(\R)}^2 /k$. It is also clear that, since $\|\mb E_{\pm\infty}\|_{L^\infty(\R)} \lesssim k$, the second term on the right in \eqref{e:detW} is $O(k^2)$. To sum up, we have shown
\begin{equation}\label{e:wyk0}
 W_{\mb Y}(k,0) = c_+c_- A(k) + O(k^2),
 \end{equation}
with $A(k)\geq A_0 k$ for some $A_0>0$ independent of $k$.

Now we turn to the perturbed operator $\mathcal L_T$. Let $U_{\pm \infty}^\lambda$ be the solutions to \eqref{e:U-lambda} as above. Applying Lemma \ref{l:approx-r} with $\lambda_1 = \lambda_2 = \lambda$,  we have
 \begin{equation}\label{e:E-R}
 \begin{split}
 \mb U_{\pm \infty}^\lambda(y) &=  \mb Y_{\pm\infty}^{\lambda}(y) +  \tilde{\mb E}_{\pm \infty}(y),
 \end{split}
 \end{equation}
 with $\|e^{ky}\tilde{\mb E}_{\infty}\|_{L^\infty([0,\infty),\R^2)}+\|e^{-ky}\tilde{\mb E}_{-\infty}\|_{L^\infty((-\infty,0],\R^2)}\lesssim \|d\|_{L^1(\R)} + \|b\|_{L^1(\R)} \lesssim \delta$.

With the Wronskian $W_{\mb U}(\lambda,y)$ defined as in the proof of Theorem \ref{t:eigenvalue}, we again write $W_{\mb U}(k,y) = W_{\mb U}(m^2-k^2,y)$, and obtain 
\begin{equation}\label{e:WU-R}
 \begin{split}
W_{\mb U}(\lambda,0) &= \det(\mb Y_\infty^\lambda(0), \mb Y_{-\infty}^\lambda(0)) + \det(\tilde {\mb E}_\infty(0), \mb Y_{-\infty}^\lambda(0)) + \det(\mb Y_\infty^\lambda(0), \tilde {\mb E}_{-\infty}(0)) + \det(\tilde {\mb E}_\infty(0),\tilde {\mb E}_{-\infty}(0))\\
 &= W_{\mb Y}(\lambda,0)  +   \det(\tilde {\mb E}_\infty(0), \mb Y_{-\infty}^\lambda(0)) + \det(\mb Y_\infty^\lambda(0), \tilde {\mb E}_{-\infty}(0)) + O(\delta^2).
\end{split}
\end{equation}
Since $\tilde {\mb E}_{\pm \infty} = U_{\pm\infty}^\lambda - Y_{\pm \infty}^\lambda$ satisfy $\tilde {\mb E}_{\pm\infty}'' = (F''(S) - m^2 - \lambda)\tilde {\mb E}_{\pm\infty} - b(U_{\pm\infty}^\lambda)' + dU_{\pm\infty}^\lambda$, we have
\[\det(\tilde {\mb E}_{\pm\infty}, \mb Y_{\mp\infty}^\lambda)'(y) = [d U_{\pm\infty}^\lambda - b(U_{\pm\infty}^\lambda)'] Y_{\mp\infty}^{\lambda}.\]
Because $b,d\in L^1(\R)$, $|U_\infty^\lambda| \lesssim e^{-ky}$, and $|Y_{-\infty}^\lambda|\lesssim e^{ky}$, the expression $\left[d  U_{\infty}^\lambda - b (U_{\infty}^\lambda)'\right] Y_{-\infty}^\lambda$ is integrable on $[0,\infty)$, and we can use \eqref{e:E-R} to write
\[ 
\begin{split}
\det( \tilde {\mb E}_\infty(0), \mb Y_{-\infty}^\lambda(0)) &= \int_{0}^\infty \left[d  U_{\infty}^\lambda - b (U_{\infty}^\lambda)'\right] Y_{-\infty}^\lambda \dd w\\
& = \int_0^\infty  [dY_\infty^\lambda - b(Y_\infty^\lambda)'] Y_{-\infty}^\lambda \dd w + \int_0^\infty [d \tilde E_\infty^1 - b \tilde E_\infty^2] Y_{-\infty}^\lambda \dd w,
\end{split}
\]
where the last integral converges and is $O(\delta^2)$ since $\|\tilde {\mb E}_\infty\| \lesssim \delta e^{-ky}$ and $Y_{-\infty}^\lambda \lesssim e^{ky}$. 
For the first integral on the right, we use Lemma \ref{l:approx} with $d=b=0$ and obtain
\[\int_0^\infty [dY_\infty^\lambda - b(Y_\infty^\lambda)']Y_{-\infty}^\lambda \dd w = \int_0^\infty [dY_\infty^{m^2} - b(Y_\infty^{m^2})']Y_{-\infty}^{m^2} \dd w + O(\delta k).\]
After applying a similar analysis to $\det(\mb Y_\infty^\lambda(0), \tilde {\mb E}_{-\infty}(0))$, the expression \eqref{e:WU-R} becomes 
\[ 
W_{\mb U}(k,0) = W_{\mb Y}(k,0) + c_+c_-\int_{-\infty}^\infty  [d(R(w))^2 - b R(w)R'(w)]\dd w + O(\delta k) +  O(\delta^2).
\]
With \eqref{e:wyk0}, this implies
\[ W_{\mb U}(k,0) = c_+c_- \left( A(k) + \int_{-\infty}^\infty [d (R(w))^2 - bR(w) R'(w)]\dd w \right) + O(\delta k) + O(\delta^2).\]
For $\delta$ small enough, the expression inside the parentheses determines whether any zeroes of $W_{\mb U}(k,0)$ are present for $k>0$. The bound $A(k) \geq A_0 k$ with $A_0>0$ implies statement (a) of the theorem.

For statement (b), the assumption that $m^2$ is not a resonance or eigenvalue implies $W_{\mb Y}(0,0) \neq 0$. The approximation \eqref{e:E-R} easily implies $W_{\mb U}(0,0) = W_{\mb Y}(0,0) + O(\delta) \neq 0$ for $\delta$ small enough.
\end{proof}

\section{Orbital stability}\label{s:orbital}

In this section, we prove orbital stability, i.e. that solutions starting close to $T$ are always close to some shifted version of $T$. 


\begin{proof}[Proof of Theorem \ref{t:orbital}]
For any solution $u$ of \eqref{e:main-y}, the energy
\[ E(u) = \int_{\R} \left[ \frac 1 2 (\partial_t u)^2 + \frac 1 2 (\partial_y u)^2 - \frac 1 2 cu^2 + F(u)\right] \omega(y)\dd y,\]
is conserved, where $\omega(y) = \exp(\int_{-\infty}^y b(z)\dd z) = 1 + O(\delta)$, uniformly in $y$. We also define the potential energy
\[ E_p(u) = \int_\R \left[\frac 1 2 (\partial_y u)^2 -\frac 1 2 cu^2+ F(u)\right] \omega(y)\dd y.\]
A simple computation shows that 
\begin{equation}\label{e:E-approx}
|E_p(\psi) - \tilde E_p(\psi)| \leq  C\delta\left(\|\psi\|_{L^\infty(\R)}^2 + \tilde E_p(\psi)\right) , \quad \psi \in H^1_T(\R),
\end{equation}
where $\tilde E_p$ is the potential energy corresponding to the constant coefficient equation \eqref{e:constant-speed}:
\[ \tilde E_p(\psi) := \int_\R \left[ \frac 1 2 (\partial_t \psi)^2 + \frac 1 2 (\partial_y \psi)^2  + F(\psi)\right] \dd y \]
The idea is to use the (known) property that $\tilde E_p(\psi) - \tilde E_p(S)$ controls the distance between $\psi$ and $S$, to show the corresponding fact for $E_p$ and $T$. In more detail, for $q>0$, define
\[ d_q(\psi,T)^2 := \inf_{\xi\in \R}\int_{\R} [(\partial_y \psi(y) - T'(y+\xi))^2 + q (\psi(y)-T(y+\xi))^2 ] \dd y, \]
for any $\psi$ in the energy space. We define $d_q(\psi,S)$ in the analogous way. Proposition 1 of \cite{henry-perez-wreszinski} proves the following: There exist $C, r, q>0$ such that 
\[ d_q(\psi,S)^2 \leq C (\tilde E_p(\psi) - \tilde E_p(S)),  \]
whenever $d_q(\psi,S) \leq r$.\footnote{Proposition 1 in \cite{henry-perez-wreszinski} is stated for $u(t,\cdot)$ where $u$ is a solution of \eqref{e:constant-speed}, but an examination of the proof shows that the conclusion holds for any $\psi(y)$ satisfying the hypotheses stated here.} Note that
\[\begin{split}
 |\tilde E_p(T) -\tilde  E_p(S)| &= \left| \int_\R \left[ \frac 1 2 (\partial_y S_b)^2 +\partial_y S \partial_y S_b+ F(S+S_b) - F(S) \right]\dd y\right|\\
&\leq  \|S_b\|_{H^1(\R)}^2 + \|\partial_y S\|_{L^2(\R)}^2 + \|F\|_{C^1([a_-,a_+]} \|S_b\|_{L^1(\R)}\\
& \lesssim \delta,
\end{split}\]
by Theorem \ref{t:exist}. Using \eqref{e:E-approx} twice, we then have
\begin{equation}\label{e:dS}
\begin{split}
 d_q(\psi,S)^2 
  &\leq C\left(E_p(\psi) - E_p(S)\right) + C\delta\left(\||\psi| + |T|\|_{L^\infty(\R)}^2 + \tilde E_p(\psi) + \tilde E_p(S)\right)\\
  &\leq C\left(E_p(\psi) - E_p(T)\right) + C\delta\left(1+ \tilde E_p(\psi) + \tilde E_p(S)\right).
 \end{split}
\end{equation}
To get to the last line, we used Sobolev embedding to write $\|\psi\|_{L^\infty(\R)} \leq C_q d_q(\psi,S)$, and combined this term into the left-hand side. 

 Since $\int_{\R} (u(t,y)-S(y+\xi))^2 \dd y \to \infty$ as $\xi\to \pm\infty$, there is some $\xi_0\in \R$ where the infimum defining $d_q(u(t,\cdot),S)$ is achieved. To save space, write $T_\xi = T(y+\xi_0)$, and similarly for $S_\xi$ and $S_{b,\xi}$. We then have
\begin{equation}\label{e:dT}
 \begin{split}
d_q(\psi,T)^2 &\leq \int_{\R}\left[ (\partial_y \psi - T'_\xi)^2 + q(\psi-T_\xi)^2\right] \dd y \\
&=  \int_{\R}\left[ (\partial_y \psi - S'_\xi)^2 +  q(\psi-S_\xi)^2\right] \dd y\\
&\quad + \int_\R \left[ (S_{b,\xi}')^2 + q S_{b,\xi}^2  -2 (\partial_y \psi - S_\xi')\partial_yS_{b,\xi} - 2q(\psi - S_\xi) S_{b,\xi}\right]\dd y\\
&\leq d_q(\psi,S)^2 + Cq\|S_b\|_{H^1(\R)}^2 + \|S_b\|_{H^1(\R)}\|\psi-S\|_{H^1(\R)}\\
&\leq d_q(\psi,S)^2 + Cq\|S_b\|_{H^1(\R)}^2 + \|S_b\|_{H^1(\R)} d_q(\psi,S)^2\\
&\leq 2 d_q(\psi,S)^2 + C\delta^2.
\end{split}
\end{equation}
For $\delta>0$ small enough compared to $r$ and $q$, this implies $d_q(\psi,T)^2 \leq 2d_q(\psi,S)^2 + r/2$. By exchanging the roles of $T$ and $S$ in this calculation, we also obtain $d_q(\psi,S)^2 \leq 2d_q(\psi,T) + r/2$. 

Next, combining \eqref{e:dS} and \eqref{e:dT}, 
\begin{equation}\label{e:dq}
\begin{split}
d_q(\psi,T)^2 &\leq C\left(E_p(\psi) - E_p(T)\right) + C\delta\left(1 + \tilde E_p(\psi) + \tilde E_p(S)\right).
\end{split}
\end{equation}
This inequality holds for $\psi$ such that $d_q(\psi,S) \leq r$. By above, we can ensure this condition by choosing $d_q(\psi,T) \leq r/4$.

Now, for a solution $u$ to \eqref{e:main-y} with $d_q(u(0,\cdot),T) \leq r/4$ and $\partial_t u(0,\cdot)$ sufficiently small in $L^2(\R)$, \eqref{e:dq} implies that 
\[ d_q(u(t,\cdot),T)^2 \leq C( E( u(t,\cdot)) - E(T)) + C\delta\left(1+\tilde E_p(u(t,\cdot)) + \tilde E_p(S)\right).\]
The quantity $E(u(t,\cdot))$ is conserved in time. Calculations similar to \eqref{e:E-approx} show that $\tilde E_p(u) \leq 2 E_p(u) + C\delta \|u\|_{L^\infty(\R)}^2\leq E(u) + C\delta d_q(u,T)^2$, and the last term may be combined into the left side. We finally have
\[ d_q(u(t,\cdot),T)^2 \leq C(E(u_0)- E(T)) + C\delta(1+\tilde E_p(S)).\]
This right-hand side is independent of $t$, which implies the solution $u$ never leaves the neighborhood of $T$ as long as it exists. As above, for every $t$, there is some $\xi = \xi(t)$ at which the infimum defining $d_q(u,T)$ is achieved. 
By standard arguments, this time-independent bound on $d_q(u(t,\cdot),T)$ combined with energy conservation implies the solution $u$ exists for all $t\in [0,\infty)$.
\end{proof}

\appendix

\section{ODE Methods}\label{s:a}

In this section, we collect some convenient facts about the solvability and asymptotics of $2\times 2$ first-order systems on $\R$.

First, we have a standard lemma on vector-valued integral equations of Volterra type:

\begin{lemma}\label{l:volterra}
For $a \in \R$ and $\mb U\in L^\infty([a,\infty),\R^2)$, the Volterra equation
\[ \mb Z(y) = \mb U(y) + \int_y^\infty K(y,w) \mb Z(w) \dd w,\]
has a unique solution in $L^\infty([a,\infty),\R^2)$, provided
\begin{equation}\label{e:mu}
 \mu := \int_a^\infty \sup_{a<y<w} \|K(y,w)\| \dd w < \infty,
 \end{equation}
where $\|\cdot\|$ is the operator norm of the matrix $K(y,w)$. 
This solution is given by the iteration
\begin{equation}\label{e:series}
 \mb Z(y) = \mb U(y) + \sum_{n=1}^\infty \int_a^\infty \cdots \int_a^\infty \prod_{i=1}^n 1_{\{y_{i-1} < y_i\}} K(y_{i-1}, y_i) \mb U(y_n) \dd y_n \cdots \dd y_1,
 \end{equation}
with $y_0 = y$. This solution satisfies
\[ \|\mb Z\|_{L^\infty([a,\infty),\R^2)} \leq e^\mu \|\mb U\|_{L^\infty([a,\infty),\R^2)}.\]

\end{lemma}

\begin{proof}
See \cite[Lemma 2.4]{schlag2010conical} for a proof of the corresponding fact for scalar-valued Volterra equations. The proof in the present vector-valued case is essentially the same, so we omit it.
\end{proof} 

Next, we address a class of linear systems that arise from the eigenvalue problems in Sections \ref{s:stationary} and \ref{s:spectrum}:

\begin{lemma}\label{l:system}
\begin{enumerate}
\item[(a)] For $k>0$, consider the system
\begin{equation}\label{e:Z}
 \mb Y'(y) = (M_1 + M_2(y) ) \mb Y(y),
 \end{equation}
where
\[  M_1 = \left(\begin{array}{cc}0 & 1\\ k^2  & 0     \end{array}\right), \quad M_2(y) = \left(\begin{array}{cc}0 & 0\\ V(y)  & -b(y)     \end{array}\right) \]
with $V, b\in L^1(\R)$. 
There exist solutions $\mb Y_{-\infty}$, $\mb Y_\infty$ defined on $\R$, such that 
\begin{equation}\label{e:infty}
 \lim_{y\to \infty} e^{ky} \mb Y_\infty(y) = \left(\begin{array}{c} 1\\  - k\end{array}\right), \quad  \lim_{y\to -\infty} e^{-ky} \mb Y_{-\infty}(y) = \left(\begin{array}{c} 1\\   k\end{array}\right),
 \end{equation}
and the bound $|\mb Y_\infty(y)| \leq Ce^{-ky}$ holds for all $y\in \R$, where the constant depends on $k$ and $\|V+b\|_{L^1(\R)}$. These solutions also satisfy the integral equations
\begin{equation}\label{e:integrals}
\begin{split}
 \mb Y_\infty(y) &= \left(\begin{array}{c} 1\\  -k\end{array}\right)e^{-ky}  -  \frac 1 2 \int_{y}^\infty (V(w)Y_\infty - b  Y_\infty'(w) ) \left(\begin{array}{c} \frac 1 k (e^{k(y-w)} - e^{-k(y-w)})\\ e^{k(y-w)} + e^{-k(y-w)}\end{array}\right)\dd w,\\
 \mb Y_{-\infty}(y) &=  \left(\begin{array}{c} 1\\  k\end{array}\right)e^{ky}  +  \frac 1 2  \int_{-\infty}^y (V(w)Y_{-\infty} - b  Y_{-\infty}'(w) ) \left(\begin{array}{c} \frac 1 k (e^{k(y-w)} - e^{-k(y-w)})\\ e^{k(y-w)} + e^{-k(y-w)}\end{array}\right)\dd w.
 \end{split}  
 \end{equation}

\item[(b)] For $k= 0$, assume in addition that $(1+|y|^2)^{1/2} V$ and  $(1+|y|^2)^{1/2}b$ lie  in $L^1(\R)$. Then there exist solutions $\mb Y_{-\infty}, \mb Y_\infty$ to \eqref{e:Z} satisfying
\begin{equation*}
 \lim_{y\to \infty} \mb Y_{\pm\infty}(y) = \left(\begin{array}{c} 1\\  0\end{array}\right),
 \end{equation*}
as well as the integral equations
\begin{equation}
\begin{split}
 \mb Y_\infty(y) &= \left(\begin{array}{c} 1\\  0\end{array}\right)  -  \int_{y}^\infty (V(w)Y_\infty - b  Y_\infty'(w) ) \left(\begin{array}{c} y-w\\ 1 \end{array}\right)\dd w,\\
 \mb Y_{-\infty}(y) &=  \left(\begin{array}{c} 1\\  0\end{array}\right)  +   \int_{-\infty}^y (V(w)Y_{-\infty} - b  Y_{-\infty}'(w) ) \left(\begin{array}{c} y-w\\ 1\end{array}\right)\dd w.
 \end{split}  
 \end{equation}
\end{enumerate}

\end{lemma}

\begin{proof}
(a) Note that the eigenvalues of $M_1$ are $\pm k$ corresponding to eigenvectors $\left(\begin{array}{c} 1\\  \mp k\end{array}\right)$. We will find a solution to the integral equation 
\begin{equation}\label{e:Yintegral}
\mb Y_\infty(y) =  \left(\begin{array}{c} 1\\  -k\end{array}\right)e^{-ky}  -  \int_{y}^\infty e^{M_1 (y-w)} M_2(w) \mb Y_\infty(w)  \dd w, \quad y\in \R,
\end{equation}
satisfying $|\mb Y_\infty(y)|\leq Ce^{-ky}$ and $\lim_{y\to \infty} e^{ky}\mb Y_\infty = \left(\begin{array}{c} 1 \\ -k\end{array}\right)$. 
By direct calculation, such $\mb Y_\infty$ also solves \eqref{e:Z}, as well as the first integral equation in \eqref{e:integrals}. 
Letting $\mb Z(y) = e^{ky} \mb Y_\infty(y)$, \eqref{e:Yintegral} is equivalent to
\begin{equation}\label{e:Z-integral}
  \mb Z(y) =  \left(\begin{array}{c} 1\\  -k\end{array}\right)  -  \int_{y}^\infty e^{M_1 (y-w)} M_2(w) e^{k(y-w)}   \mb Z(w)  \dd w, \quad y\in \R.
  \end{equation}
By diagonalizing $M_1$, we obtain
\[ e^{M_1(y-w)}M_2 = \frac 1 2\left(\begin{array}{cc} \frac 1 k V (e^{k(y-w)} - e^{-k(y-w)}) & - \frac 1 k b (e^{k(y-w)} - e^{-k(y-w)})\\ 
 V (e^{k(y-w)} + e^{-k(y-w)}) & -b (e^{k(y-w)} + e^{-k(y-w)})\end{array}\right) .\]
With $K(y,w) := e^{M_1(y-w)} M_2(w)e^{k(y-w)}$, we therefore have
\[\|K(y,w)\|\leq C (1 + e^{2k(y-w)}) (|V(w)| + |b(w)|),\] 
and that
\[  \int_0^\infty \sup_{0<y<w} \|K(y,w)\| \dd w  \leq C (\|V\|_{L^1(\R)} + \|b\|_{L^1(\R)}). \]
Lemma \ref{l:volterra} now implies a solution to \eqref{e:Z-integral} exists on $[0,\infty)$, 
and $\|\mb Z\|_{L^\infty([0,\infty),\R^2)}$ is bounded by a constant, which implies the boundary condition \eqref{e:infty} holds for $\mb Y_\infty$, as well as the upper bound
\[  |\mb Y_\infty(y)| \leq Ce^{-ky}, \quad y\geq 0,\]
where $\mb Y_\infty = (Y_\infty, Y_\infty')$.  Applying a similar argument with $-y$ replacing $y$, we can obtain a solution $\mb Y_{-\infty}$ defined on $\R$ with 
\[ \lim_{y\to -\infty} e^{-ky} \mb Y_{-\infty}(y) = \left(\begin{array}{c} 1\\   k\end{array}\right) \quad \text{and} \quad |Y_{-\infty}(y)| + |Y_{-\infty}'(y)| \leq C e^{ky}, \quad y\leq 0.  \]
For $y< 0$, we can write
\[ Y_{\infty}(y) = c_0Y_{-\infty}(y) \left(\int_{0}^y\frac{\exp\left(-\int_{-\infty}^w b(z) \dd z\right) }{Y_{-\infty}^2(w)} \dd w + c_1\right), \]
with $c_0, c_1$ chosen so that $Y_\infty(0)$ and $Y_\infty'(0)$ match our previous definition. This formula implies $\mb Y_\infty = (Y_\infty, Y_\infty')$ solves \eqref{e:Z} and satisfies $|\mb Y_{\infty}(y)| \leq C e^{-ky}$ for negative $y$ also. By a similar method, we extend $\mb Y_{-\infty}$ to the real line and obtain $|\mb Y_{-\infty}(y)| \leq Ce^{ky}$ for all $y\in \R$.

(b) In the case $k=0$, we have
\[ e^{M_1 (y-w)} M_2(w) = \left(\begin{array}{cc} 1 & y-w\\ 0 & 1\end{array}\right)\left(\begin{array}{cc} 0 & 0 \\ V(y) & -b(y)\end{array}\right) =  \left(\begin{array}{cc} (y-w)V(y) & -(y-w)b(y) \\ V(y) & -b(y)\end{array}\right),\]
and the integral equation \eqref{e:Yintegral} reduces to 
\[ \mb Y_\infty(y) = \left(\begin{array}{c} 1\\  0\end{array}\right)  -  \int_{y}^\infty (V(w)Y_\infty - b  Y_\infty'(w) ) \left(\begin{array}{c} y-w\\ 1 \end{array}\right)\dd w.\]
Defining $K(y,w) = e^{M_1(y-w)} M_2(w)$ , we have
\[ \|K(y,w)\| \leq C \sqrt{1 + (y-w)^2} (|V(w)+|b(w)|), \]
and 
\[  \int_0^\infty \sup_{0<y<w} \|K(y,w)\| \dd w  \leq C (\|(1+|y|^2)^{1/2}V\|_{L^1(\R)} + \|(1+|y|^2)^{1/2}b\|_{L^1(\R)}). \]
By Lemma \ref{l:volterra}, a solution $\mb Y_\infty$ exists on $[0,\infty)$, which also solves \eqref{e:Z} by a direct calculation. Applying a similar method for $\mb Y_{-\infty}$ and extending both solutions to the real line proceeds as in the proof of (a).
\end{proof}

\bibliographystyle{acm}
\bibliography{sineGordon}
\end{document}